\documentclass[runningheads,numbook]{svjour3}
\usepackage{amsmath,amssymb}
\usepackage{ntheorem}
\usepackage{cite}
\usepackage{color}
\usepackage{hyperref}
\hypersetup{
    colorlinks=true,
    linkcolor=blue,     citecolor=red,     urlcolor=blue  }
\usepackage[colorinlistoftodos]{todonotes}
\usepackage{mathptmx}   %Times New Roman
\usepackage[weather,alpine,misc,geometry]{ifsym}
\usepackage{amsfonts}
\usepackage{graphicx}%
\usepackage{relsize}%
\usepackage[capitalise,nameinlink]{cleveref}
\crefformat{equation}{(#2#1#3)}
\crefmultiformat{equation}{(#2#1#3)}%
{ and~(#2#1#3)}{, (#2#1#3)}{ and~(#2#1#3)}
\crefrangeformat{equation}{(#3#1#4)--(#5#2#6)}
\crefrangemultiformat{equation}{(#3#1#4)--(#5#2#6)}%
{ and~(#3#1#4)--(#5#2#6)}{, (#3#1#4)--(#5#2#6)}{ and~(#3#1#4)--(#5#2#6)}
\newcommand{\al}{\alpha}
\newcommand{\be}{\beta}
\newcommand{\ga}{\gamma}
\newcommand{\la}{\lambda}
\newcommand{\de}{\delta}
\newcommand{\eps}{\varepsilon}
\newcommand{\bx}{\bar x}
\newcommand{\by}{\bar y}

\newcommand{\iv}{^{-1} }

\newcommand {\R} {\mathbb R}
\newcommand {\N} {\mathbb N}

\newcommand {\B} {\mathbb B}
\newcommand {\Sp} {\mathbb S}
\newcommand {\gph} {{\rm gph}\,}%Graph

\newcommand {\Int} {{\rm int}\,}

\newcommand{\toto}{\rightrightarrows}% geht nur mit amssymb.sty
\def\nbh{neighbourhood}
\def\es{\emptyset}

\def\SVM{set-valued mapping}

\def\Fr{Fr\'echet}
\newcommand{\norm}[1]{\left\Vert#1\right\Vert}

\newcommand{\set}[1]{\left\{#1\right\}}

\newcommand{\ang}[1]{\left\langle #1 \right\rangle}
\newcommand{\qdtx}[1]{\quad\mbox{#1}\quad}
\newcommand{\AND}{\quad\mbox{and}\quad}
\newcounter{mycount}

\newcommand{\rg}{{\rm rg}\,F(\bx,\by)}

\newcommand{\lip}{{\rm lip}\,F(\bx,\by)}

\newcommand\xqed{%
  \leavevmode\unskip\penalty9999 \hbox{}\nobreak\hfill
  \quad\hbox{$\triangleleft$}}

\smartqed

\newcommand\Rad[2]{{\rm rad[{#1}]}_{{\rm #2}}F(\bx,\by)}
\renewcommand{\rg}{{\rm{rg}}\,F(\bx,\by)}
\newcommand{\rga}{{\rm{rg^+}}\,F(\bx,\by)}

\renewcommand{\lip}{{\rm{lip}}\,}

\newcommand{\Bmv}{\;\Big\vert\;\;}

\begin{document}
\relsize{+1}
\title{The Radius of Metric Regularity
%in Infinite Dimensions
{Revisited}
\thanks{This note is dedicated to the memory of our friend and colleague Professor Asen Dontchev
\if{
The second author benefited
from the support of the Australian Research Council, project DP160100854, the European Union's Horizon 2020 research and innovation
programme under the Marie Sk{\l }odowska--Curie Grant Agreement No. 823731
CONMECH.
}\fi
} }
\author{
Helmut Gfrerer
\and
Alexander Y. Kruger}
%\titlerunning{Radius of Metric Regularity in Infinite Dimensions}

\institute{
H. Gfrerer \at
Institute of Computational Mathematics, Johannes Kepler University Linz, A-4040 Linz, Austria\\
\email{helmut.gfrerer@kju.at},
ORCID 0000-0001-6860-102X
\and
A. Y. Kruger (corresponding author)\at
%Centre for Informatics and Applied Optimization, Federation University, Ballarat, Australia\\
%\email{a.kruger@federation.edu.au}\\
%RMIT University, Melbourne, Australia\\
%\email{alexander.kruger@rmit.edu.au}
Optimization Research Group, Faculty of Mathematics and Statistics, Ton Duc Thang
University, Ho Chi Minh City, Vietnam\\
\email{alexanderkruger@tdtu.edu.vn},
ORCID 0000-0002-7861-7380
}
\maketitle

\begin{abstract}
The paper extends the radius of metric regularity theorem by Dontchev, Lewis \& Rockafellar (2003) by providing an exact formula for the radius with respect to Lipschitz continuous perturbations in general Asplund spaces, thus, answering affirmatively an open question raised twenty years ago by Ioffe.
In the non-Asplund case, we give a natural upper bound for the radius complementing the conventional lower bound in the theorem by Dontchev, Lewis \& Rockafellar.
\end{abstract}

\section{Introduction}
\label{intro}
Study of the ``radius of good behaviour'' was initiated in 2003 by Dontchev, Lewis \& Rockafellar \cite{DonLewRoc03}.
They aimed at quantifying the ``distance" from a given \emph{well-posed} problem to the set of ill-posed problems of the same kind.
The topic is obviously about stability of problems with respect to perturbations of the problem data, but it goes further than just establishing stability; the goal is to provide quantitative estimates (ideally exact formulas) of how far the problem can be perturbed before well-posedness is lost.
This is of significance, e.g., for computational methods.

It is common to describe ``good behaviour'' of problems in terms of certain regularity properties of (set-valued) mappings involved in modelling of the problems, and talk about the \emph{radius of regularity}.
Not surprisingly, the first \emph{radius theorems} were established in \cite{DonLewRoc03} for the fundamental property of \emph{metric regularity}, followed in \cite{DonRoc04} by the corresponding statement for the \emph{strong metric regularity}.
The definitions of the mentioned properties are collected below (cf. \cite{RocWet98,DonRoc14,Iof17}).
Here $F:X\rightrightarrows Y$ is a \SVM, and $F\iv:Y\rightrightarrows X$ denotes its inverse, i.e., $F^{-1}(y):=\{x\in X\mid y\in F(x)\}$ for all $y\in Y$.
%For the definition of a single-valued localization of a \SVM\ we refer the readers to \cite{DonRoc14}.

\begin{definition}
\label{D1.1}
Let $X$ and $Y$ be metric spaces, $F:X\rightrightarrows Y$, and $(\bx,\by)\in\gph F$.
\begin{enumerate}
\item
$F$ is \emph{metrically regular} at $(\bx,\by)$ if there exist numbers $\al>0$ and $\de>0$ such that
\begin{equation}
\label{D1.5-1}
\al d(x, F^{-1}(y)) \leq d(y, F(x))
\quad \mbox{for all} \quad x\in B_\de(\bx),\; y\in B_\de(\by).
\end{equation}
\item
$F$ is \emph{strongly metrically regular} at $(\bx,\by)$
if it is metrically regular at $(\bx,\by)$, and $F\iv$ has a single-valued localization around $(\by,\bx)$,
i.e., there exist \nbh s $U$ of $\bx$ and $V$ of $\by$ and a function $\phi:V\to U$ such that $\gph\phi=\gph F\iv\cap(V\times U)$.
\end{enumerate}
\end{definition}

The (possibly infinite) supremum of all $\al$ satisfying \cref{D1.5-1} for some $\de>0$ is called the \emph{regularity modulus} of $F$ at $(\bx,\by)$ and is denoted by $\rg$.
It is also known as the \emph{modulus} (or \emph{rate}) of \emph{surjection} \cite{Iof17}, and
%that there holds
$\rg =1/{\rm reg\,}F(\bx\,|\,\by)$,
where ${\rm reg\,}F(\bx\,|\,\by)$ is the regularity modulus employed in \cite{DonLewRoc03,DonRoc04,DonRoc14}.
\if{
Thus,
\begin{gather}
\label{rg}
\rg=\liminf_{(x,y)\to(\bx,\by),\;x\notin F\iv(y)}
\frac{d(y,F(x))}{d(x,F^{-1}(y))}.
\end{gather}
}\fi
The case $\rg=0$ (or equivalently, ${\rm reg\,}F(\bx\,|\,\by)=+\infty$) indicates the absence of metric regularity.

In \cite{DonLewRoc03,DonRoc04},
the authors considered perturbations of a \SVM\ over the classes $\mathcal{F}_{lin}$ of affine (linear)  functions and $\mathcal{F}_{lip}$ of single-valued functions $f:X\to Y$, which are Lipschitz continuous near the reference point $\bx$, and used the \emph{Lipschitz modulus}
$$
\lip f(\bx)= \limsup_{\substack{x,x'\to\bx,\,x\neq x'}}     \frac{d(f(x),f(x'))}{d(x,x')}
$$
to measure the size of perturbations.

The next
%radius
theorem combines \cite[Theorem~1.5]{DonLewRoc03} and \cite[Theorems~4.6]{DonRoc04}; see also \cite[Theorems~6A.7 and 6A.8]{DonRoc14}.

\begin{theorem}
%[Radius theorem]
\label{T1.2}
Let $X$ and $Y$ be Banach spaces, $F:X\rightrightarrows Y$, $(\bx,\by)\in\gph F$, and $\gph F$ be closed near $(\bx,\by)$.
The following estimates hold true:
\begin{gather}
\label{T1.2-1}
\Rad{R}{lin}\geq
\Rad{R}{lip}\geq\rg,
\end{gather}
If $F$ is strongly metrically regular at $(\bx,\by)$, then
\begin{gather}
\label{T1.2-1a}
\Rad{sR}{lin}\geq
\Rad{sR}{lip}\geq\rg.
\end{gather}
If $\dim X<\infty$ and $\dim Y<\infty$, then all the inequalities in \cref{T1.2-1,T1.2-1a} hold as equalities.
Moreover, the equalities remain valid if $\mathcal{F}_{lin}$ is restricted to affine functions of rank $1$.
\end{theorem}

$\Rad{R}{lip}$ in \cref{T1.2-1} stands for
the radius of metric regularity of $F$ at $(\bx,\by)\in\gph F$ over the class $\mathcal{F}_{lip}$ of Lipschitz continuous perturbations:
\begin{align}
\label{rad}
\Rad{R}{lip}:=& \inf_{f\in\mathcal{F}_{lip}}\{\lip f(\bx) \mid\;\;
F+f \mbox{ is not metrically regular at } (\bx,\by)\}.
\end{align}
$\Rad{R}{lin}$ corresponds to replacing $\mathcal{F}_{lip}$ in \cref{rad} with the class $\mathcal{F}_{lin}$ of affine perturbations, while $\Rad{sR}{lip}$ and $\Rad{sR}{lin}$ in \cref{T1.2-1a} are defined in a similar way for the property of strong metric regularity (`sR' for brevity).
We assume everywhere without loss of generality that perturbation functions $f$
%in the above definitions
satisfy
${f(\bx)=0}$.
(Thus, the functions in $\mathcal{F}_{lin}$ are actually linear.)
For a more general definition of the radius $\Rad{Property}{\mathcal{P}}$ allowing for other properties and other classes of perturbations, we refer the readers to \cite{GfrKru}.
\sloppy

The critical inequality $\Rad{R}{lip}\geq\rg$ in \cref{T1.2-1} has its roots in the fundamental theorems of Lyusternik and
Graves, and is sometimes referred to as the \emph{extended Lyusternik--Graves theorem}.
The latter theorem has a long and rather well known history; cf. \cite{DmiMilOsm80,Iof00,DonLewRoc03,DonRoc04,Mor06.1, DmiKru08,DonRoc14,Iof17}.
\if{
 The following version is from \cite[Theorem 3.3]{DonLewRoc03}.

\begin{theorem}[Extended Lyusternik--Graves theorem]
%\label{ThExtGraves}
Let $X$ and $Y$ be Banach spaces, $F:X\rightrightarrows Y$, $f:X\to Y$, $(\bx,\by)\in\gph F$, and $\gph F$ be closed near $(\bx,\by)$.
If $\rg >\alpha$, $\lip f(\bx) < \lambda$, and $\la< \alpha$, then
\[{\rm rg\,}(F+f)\big(\bx, \by+f(\bx)\big)>\alpha-\lambda.\]
\end{theorem}

From this theorem one immediately obtains the following relation between the regularity and Lipschitz moduli; cf. \cite[Corollary 2.4]{DonRoc04}, \cite[Theorem~1]{DmiKru08}.
}\fi
The version below relates the regularity moduli of \SVM s and the Lipschitz modulus of the perturbation function; cf. \cite[Corollary 2.4]{DonRoc04}, \cite[Theorem~1]{DmiKru08}.

\begin{theorem}\label{ThBndReg}
Let $X$ and $Y$ be Banach spaces, $F:X\rightrightarrows Y$, $(\bx,\by)\in\gph F$, $\gph F$ be closed near $(\bx,\by)$, and let $f\in\mathcal{F}_{lip}$.
%If $\rg<\infty$ and $\lip f(\bx)<\rg$, then
Then
\begin{equation}
\label{EqBndRg}
{\rm rg}\,(F+f)(\bx,\by)\geq \rg -\lip f(\bx).
\end{equation}
%If $\rg=\infty$ then ${\rm rg}\,(F+f)(\bx,\by)=\infty$ holds for any $f\in \mathcal{F}_{lip}$.
\end{theorem}

There have been several attempts to study stability of metric regularity with respect to set-valued perturbations under certain assumptions either of sum-stability or on ways of measuring the distance between \SVM s \cite{Iof00,Iof01,NgaThe08,NgaTroThe14, AdlCibNga15,HeXu22}.

In finite dimensions, \cref{T1.2} gives exact formulas for the radii of the two regularity properties from \cref{D1.1} in terms of the regularity modulus, while in general Banach spaces it provides lower bounds for the radii and hence, sufficient conditions for the stability of the properties.
This observation naturally raises the  following important question:
\begin{enumerate}
\item[(A)]
{\em Can any of the inequalities in \cref{T1.2-1,T1.2-1a} hold as equalities in infinite dimensions?}
\end{enumerate}
Closely related
%with this question
is the following question posed by Ioffe \cite{Iof03}:
\begin{enumerate}
\item[(B)]
{\em Is the bound \eqref{EqBndRg} sharp?
In other words, in the setting of \cref{ThBndReg}, if ${\rg<+\infty}$ and $r\in[0,\rg]$, is there a function $f\in\mathcal{F}_{lip}$ such that $\lip f(\bx)=r$ and ${\rm rg}\,(F+f)(\bx,\by)= \rg -r$?}
\sloppy
\end{enumerate}
Note that, from a positive answer to the latter question, the equality $\Rad{R}{lip}=\rg$ follows by taking $r=\rg$.

A partial positive answer to question (A) was given already in
%the original paper
\cite{DonLewRoc03} (see also \cite[Theorems~6A.2]{DonRoc14}): equalities hold in \cref{T1.2-1} when $F$ is positively homogeneous. For the special case when $Y$ is finite dimensional, equalities in \cref{T1.2-1} were shown for a certain class of mappings in \cite{Mor04}.
Next, equalities in \cref{T1.2-1} were established in \cite{CanDonLopPar05} for a special mapping defined by a semi-infinite
system of linear equalities and inequalities.
However, this is not the case for general \SVM s: as shown in Ioffe \cite{Iof03.2} (see also \cite[Theorem~5.61]{Iof17}), the inequality $\Rad{R}{lin}\geq\rg$ in \cref{T1.2-1} can be strict even when $X=Y$ is a Hilbert space, and $F$ is a single-valued function having reasonably good differentiability properties.

Question (B) was positively answered by Ioffe \cite{Iof03} for general set-valued mappings in the finite dimensional setting and for single-valued continuous functions from a metric space $X$ into a Banach space $Y$. With respect to the latter result, Ioffe  stated that ``\emph{The question of whether or not a similar fact is
valid for set-valued mappings remains open}.''

To the best of our knowledge, there has been no further progress in addressing the two questions stated above.
In the current note we make another step to close the gap.
We show that in Asplund spaces
the bound \eqref{EqBndRg} is sharp for general closed graph set-valued mappings, thus, giving a positive answer to question (B) and showing the equality $\Rad{R}{lip}=\rg$.
We also obtain the relation $\Rad{sR}{lip}=\rg$
in the case when $F$ is strongly metrically regular. In the non-Asplund case, we provide natural upper bounds for the radius complementing the conventional lower bound in \cref{T1.2-1,T1.2-1a}.

In the aforementioned paper \cite{DonRoc04} by Dontchev and Rockafellar, besides metric regularity and strong metric regularity, the properties of metric subregularity and strong metric subregularity were considered.
It was shown that the radius of strong metric subregularity (under calm perturbations) in finite dimensions follows the same pattern as that of (strong) metric regularity, i.e., it equals the modulus of metric subregularity.
However, the radius of (not strong) metric subregularity fails to satisfy the paradigm promoted in \cite{DonLewRoc03, DonRoc04}, and the property requires new approaches.
In the recent papers \cite{DonGfrKruOut20,GfrKru}, the radius of metric subregularity has been analyzed for various classes of perturbations in finite and infinite dimensions.
Lower and upper bounds for the radius have been established which are different from the modulus of metric subregularity, and can
differ from each other by a factor of at most two.
The radius of strong metric subregularity has also been examined in \cite{GfrKru}, and the corresponding result in \cite{DonRoc04} has been extended to  infinite dimensions.

After some preliminaries in the next \cref{Pre}, the main results are formulated in \cref{S3}. \cref{T3.1} states that the estimate \eqref{EqBndRg} is precise in the Asplund space setting.
%The  main radius
\cref{T3.2} provides the missing equality $\Rad{R}{lip}=\rg$ in the Asplund space setting, and combines it with the other estimates for the radius.
The main tools used in the proofs of these theorems are encapsulated in a separate \cref{L3.1}. It gives a little more general relations which can be of independent interest.
The proof of \cref{L3.1} makes a separate \cref{S4}. It is partially based on our recent work on the radius of (strong) metric subregularity \cite{DonGfrKruOut20,GfrKru}.

\section{Preliminaries}
\label{Pre}

The note follows the style and (rather self-explanatory) notation of \cite{GfrKru}.
$X$ and $Y$ are normed spaces.
In most statements, they are additionally assumed to be Banach or even Asplund.
Their topological duals are denoted by $X^*$ and $Y^*$, respectively, while $\langle\cdot,\cdot\rangle$ denotes the bilinear form defining the pairing between the spaces.
Recall that a Banach space is \emph{Asplund} if every continuous convex function on an open convex set is Fr\'echet differentiable at all points of a dense subset of its domain, or equivalently, if the dual of each
%its
separable subspace is separable \cite{Phe93}.
All reflexive, particularly, all finite dimensional Banach spaces are Asplund.

The \emph{open} unit balls in a normed space and its dual are denoted by $\B$ and $\B^*$, respectively, while $\Sp$ and $\Sp^*$ stand for the unit spheres (possibly with a subscript denoting the space).
$B_\de(x)$ and $\overline B_\de(x)$ denote, respectively, the
{open} and closed
balls with radius $\de>0$ and centre $x$.
Norms and distances (including point-to-set distances) in all spaces are denoted by the same symbols
$\|\cdot\|$ and $d(\cdot,\cdot)$, respectively.
A subset $\Omega\subset X$ is said to be closed \emph{near a point} $\bx\in\Omega$ if $\Omega\cap U$ is closed for some closed \nbh\ $U$ of $\bx$.
Symbols $\R$, $\R_+$ and $\N$ denote the sets of all real numbers, all nonnegative real numbers and all positive integers, respectively, and $\R_\infty:=\R\cup\{+\infty\}$.
We use the following conventions:
$\inf\es_{\R}=+\infty$ and
$\sup\es_{\R_+}=0$, where $\es$ (possibly with a subscript) denotes the empty subset (of a given set).

Products of primal and dual normed spaces are assumed to be equipped with the sum and maximum norms, respectively:
\begin{gather*}
%\label{norm}
\norm{(x,y)}=\norm{x}+\norm{y}, \quad (x,y)\in X\times Y,
\\
%\label{dnorm}
\notag
\norm{(x^*,y^*)}=\max\{\norm{x^*},\norm{y^*}\}, \quad (x^*,y^*)\in X^*\times Y^*.
\end{gather*}

Given a subset $\Omega\subset X$, a point $\bx\in\Omega$, and a number $\eps>0$, the set
\begin{gather}
\label{nc}
N_{\Omega,\eps}(\bx):=\left\{x^*\in X^*\,\Big |\,\limsup_{\Omega\ni x\to\bx,\; x\ne\bx}\frac{\ang{x^*,x-\bx}}{\norm{x-\bx}}< \eps\right\}
\end{gather}
is called the {\em set of \Fr\ $\eps$-normals} to $\Omega$ at $\bx$,
%If $\eps=0$, the set \cref{nc} reduces to the \Fr\ normal cone $N_{\Omega}(\bx)$.
while $N_\Omega(\bx)= \bigcap_{\epsilon>0}N_{\Omega,\epsilon}(\bx)$ is the {\em \Fr\ normal cone} to $\Omega$ at $\bx$.

The graph of a \SVM\
$F:X\rightrightarrows Y$ is defined as
$\gph F:= \{ (x,y)\in X\times Y\mid y \in F(x) \}$.
Given a point $(\bx,\by)\in\gph F$, and a number $\eps\ge0$, the mapping $D^*_\eps F(\bx,\by):Y^*\toto X^*$ defined for all $y^*\in Y^*$ by
\begin{gather}
\label{cd}
D^*_\eps F(\bx,\by)(y^*):=\{x^*\in X^*\mid (x^*,-y^*)\in N_{\gph F,\eps}(\bx,\by)\},
\end{gather}
is called the {\em \Fr\ $\eps$-coderivative}
of $F$ at $(\bx,\by)$.
If $\eps=0$, it reduces to the {\em \Fr\ coderivative} ${D}^*F(\bx,\by)$
(with the convention in \eqref{cd} that $N_{\gph F,0}(\bx,\by):=N_{\gph F}(\bx,\by)$).

\begin{remark}
The set of \Fr\ $\eps$-normals is often defined \cite{Kru81.1} with the non-strict inequality in \eqref{nc}.
Then it allows also for the case $\eps=0$ in \cref{nc} directly.
Whether the strict or non-strict inequality is used in definition \eqref{nc} does not affect definition \eqref{rga} and the estimates in the current paper, but, as observed by a reviewer, using the strict inequality in \eqref{nc} leads to slight simplifications in some proofs.
\end{remark}

Employing \cref{cd}, we define another nonnegative quantity characterizing the behaviour of $F$ near $(\bx,\by)$ and closely related to the regularity modulus $\rg$:
\begin{gather}
\label{rga}
\rga:=\liminf_{\substack{\gph F\ni(x,y)\to(\bx,\by),\;\eps\downarrow0\\ x^*\in{D}_\eps^*F(x,y)(S_{Y^*})}}\|x^*\|
=\sup_{\eps>0}\;\inf_{\substack{(x,y)\in\gph F\cap B_\eps(\bx,\by)\\ x^*\in{D}_\eps^*F(x,y)(S_{Y^*})}}\|x^*\|.
\end{gather}
Observe that
\begin{gather}
\label{rgb}
\rga\le\liminf_{\substack{\gph F\ni(x,y)\to(\bx,\by)\\ x^*\in{D}^*F(x,y)(S_{Y^*})}}\|x^*\|.
\end{gather}

\begin{lemma}
\label{Lem1}
Let $X$ and $Y$ be Banach spaces,
$F:X\rightrightarrows Y$,
and $(\bx,\by)\in\gph F$.
Then $\rg\le\rga$.

If $X$ and $Y$ are Asplund, and $\gph F$ is closed near $(\bx,\by)$, then $\rg=\rga$.
\end{lemma}

\begin{proof}
If $\rg=0$ or $\rga=+\infty$, then we trivially have $\rg\le\rga$.
Thus, we may assume that $\rg>0$ and $\rga<\infty$. Pick any $\al\in(0,\rg)$ and $\be\in(\rga,+\infty)$.
Then condition \cref{D1.5-1} is satisfied for some $\de>0$, and, given an arbitrary number $\xi\in(0,\de)$, there exist $(x,y)\in\gph F\cap B_\xi(\bx,\by)$, $y^*\in S_{Y^*}$, and $x^*\in{D}_{\xi}^*F(x,y)(y^*)$ such that $\|x^*\|<\be$.
Consider a sequence $\{v_{k}\}$ in $Y$ such that $y\ne v_{k}\to y$ and $\ang{y^*,\frac{v_{k}-y}{\|v_{k}-y\|}}\to-1$ as $k\to\infty$.
Then, for each sufficiently large $k\in\N$, we have $v_{k}\in B_\de(\by)$ and,
by \cref{D1.5-1}, one can find a $u_{k}\in F\iv(v_k)$ such that
$\al\|u_{k}-x\|<(1+1/k)\|v_{k}-y\|,$
and consequently,
$$\|v_{k}-y\|\le\|(u_{k},v_{k})-(x,y)\|\le ((1+1/k)\al\iv+1)\|v_{k}-y\|.$$
By definition of the \Fr\ $\eps$-coderivative, using the notation $\mu_+:=\max\{\mu,0\}$, we obtain:
\begin{align*}
\xi&>\limsup_{k\to\infty} \frac{\left(\ang{x^*,u_k-x}-\ang{y^*,v_k-y}\right)_+} {\|(u_k,v_k)-(x,y)\|}\ge \limsup_{k\to\infty} \frac{\left(\|v_k-y\|-\|x^*\|\|u_k-x\|\right)_+} {((1+1/k)\al\iv+1)\|v_{k}-y\|}
\\&\ge
\limsup_{k\to\infty} \frac{1-\|x^*\|(1+1/k)\al\iv} {(1+1/k)\al\iv+1}= \frac{1-\|x^*\|\al\iv} {\al\iv+1}> \frac{\al-\be} {\al+1}.
\end{align*}
Since $\xi>0$ can be arbitrarily small, it follows that $\al\le\be$.
Letting $\al\uparrow\rg$ and $\be\downarrow\rga$, we arrive at $\rg\le\rga$.
The equality $\rg=\rga$ in the Asplund space setting is a consequence of \cite[Theorem~4.5]{Mor06.1} and \cref{rgb}.
\qed\end{proof}

\begin{remark}
The above proof of inequality $\rg\le\rga$ is a modification of the corresponding parts of the proofs of \cite[Theorem~3.1 and Theorem~5.1\,(i)]{Kru88}.
It can also be deduced from \cite[Theorem~1.43(i)]{Mor06.1}.
%The inequality $\rg\le\rgb$ is contained in \cite[Theorem 1.54(i)]{Mor06.1}.
In view of \cite[Theorem~4.5]{Mor06.1}, it follows from \cref{Lem1} that
\cref{rgb} becomes equality if the spaces are Asplund.
\end{remark}

The next statement from \cite{GfrKru}
%complements \cref{T2.2}(ii).
%It
extends \cite[Theorem 1.62\,(i)]{Mor06.1} which addresses the case $\eps=0$.

\begin{lemma}
\label{T2.3}
Let
%$X$ and $Y$ be normed spaces,
$F:X\rightrightarrows Y$,
$f:X\to Y$,
$(\bx,\by)\in\gph F$, and $\eps\ge0$.
Suppose that $f$ is \Fr\ differentiable at $\bx$.
Set $\eps_1:=(\|\nabla f(\bx)\|+1)\iv\eps$ and $\eps_2:=(\|\nabla f(\bx)\|+1)\eps$.
Then, for all $y^*\in Y^*$, it holds
\begin{gather*}
D^*_{\eps_1} F(\bx,\by)(y^*)\subset
D^*_\eps(F+f)(\bx,\by+f(\bx))(y^*)-\nabla f(\bx)^*y^*\subset
D^*_{\eps_2} F(\bx,\by)(y^*).
\end{gather*}
\end{lemma}

\section{Radius of metric regularity}
\label{S3}

We start this section with a lemma which constitutes a key ingredient for the proofs of our main results.
Recall from \cite[p.~552] {Iof03} that a function $f:X\to Y$ between Banach spaces
%$X,Y$
is {\em Lipschitz rank one} on an open subset $U\subset X$ if, for any $x\in U$, there is a neighborhood $U_{x}\subset U$ of $x$, on which $f$ can be represented  in the form
\[f(u)=\xi(u)y\quad (u\in U_{x}),\]
where $\xi:U_{x}\to\R$ is Lipschitz continuous and $y\in Y$.

\begin{lemma}
\label{L3.1}
Let $X$ and $Y$ be Banach spaces, $F:X\rightrightarrows Y$, $(\bx,\by)\in\gph F$, and $\gph F$ be closed near $(\bx,\by)$.
Suppose that $\rga<+\infty$.
Then there exists a function $f\in\mathcal{F}_{lip}$, Lipschitz rank one on $X\setminus\{\bx\}$, such that
$\lip f(\bx)\leq\rga$, and
$${\rm rg}^+(F+\alpha f)(\bx,\by)\leq(1-\alpha)\rga
\qdtx{for all}\alpha\in[0,1];$$
in particular, ${\rm rg}^+(F+f)(\bx,\by)=0$.
\end{lemma}

The proof of \cref{L3.1} is given in the next section.
We are now in a position to extend \cite[Theorem 4.1]{Iof03} to set-valued mappings and, thus, to answer positively question (B) in \cref{intro} in the Asplund space setting.

\begin{theorem}
\label{T3.1}
Let $X$ and $Y$ be Asplund spaces, $F:X\rightrightarrows Y$, $(\bx,\by)\in\gph F$, and $\gph F$ be closed near $(\bx,\by)$.
If $\rg<+\infty$, then, for every real $r\in [0,\rg]$, there is a function $f\in\mathcal{F}_{lip}$, Lipschitz rank one on $X\setminus\{\bx\}$, such that $\lip f(\bx)=r$ and
  \[{\rm rg\,}(F+f)(\bx,\by)=\rg-r.\]
\end{theorem}

\begin{proof}
If $\rg=0$, then $r=0$, and we take $f\equiv0$.
Let $\rg<+\infty$ and $r\in [0,\rg]$.
Set $\alpha:=r/\rg$.
Thus, $\alpha\in[0,1]$.
By \cref{Lem1}, $\rga=\rg<+\infty$, and by \cref{L3.1}, there is a function $\tilde f\in\mathcal{F}_{lip}$, Lipschitz rank one on $X\setminus\{\bx\}$, such that
$\lip\tilde f(\bx)\leq\rga$, and
$${\rm rg}^+(F+\alpha \tilde f)(\bx,\by)\leq(1-\alpha)\rga=\rg-r.$$
The function $f:=\alpha\tilde f$ is also Lipschitz rank one on $X\setminus\{\bx\}$, and
$$\lip f(\bx)=\alpha\lip\tilde f(\bx)\leq \alpha\rga=\alpha\rg=r.$$
Since $f$ is continuous near $\bx$, the graph of $F+f$ is closed near $(\bx,\by)$ and, by \cref{Lem1}, ${\rm rg}^+(F+f)(\bx,\by)={\rm rg\,}(F+f)(\bx,\by)$.
Thus,
\[{\rm rg}\,(F+f)(\bx,\by)={\rm rg}^+(F+\alpha \tilde f)(\bx,\by)\leq\rg-r.\]
On the other hand, by \cref{ThBndReg}, ${\rm rg}\,(F+f)(\bx,\by)\geq\rg -\lip f(\bx)\geq\rg-r$.
Hence, $\lip f(\bx)=r$ and ${\rm rg\,}(F+f)(\bx,\by)=\rg-r$.
%The proof is complete.
\qed\end{proof}

\begin{remark}
The single-valued version of \cref{T3.1} in \cite[Theorem 4.1]{Iof03} claims the existence of a perturbation function with the properties as in \cref{T3.1} and being Lipschitz rank one on the whole space $X$.
However, a close look at the proof of \cite[Theorem 4.1]{Iof03} shows that the function constructed there is proved to be Lipschitz rank one only on $X\setminus\{\bx\}$.
\end{remark}

Below is our main radius theorem for metric regularity.
It employs regularity constant \cref{rga} as an upper bound for the radius of metric regularity and extends \cref{T1.2}.

\begin{theorem}
\label{T3.2}
Let $X$ and $Y$ be Banach spaces,
$F:X\rightrightarrows Y$,
$(\bx,\by)\in\gph F$, and $\gph F$ be closed near $(\bx,\by)$.
Then
\begin{equation}
\label{EqRad1}
\rg \leq \Rad{R}{\lip}\leq\rga.
\end{equation}
If $X$ and $Y$ are Asplund, then all inequalities in \eqref{EqRad1} hold as equalities.\\
If, in addition,  $F$ is strongly metrically regular near $(\bx,\by)$ then we also have
\[\Rad{sR}{\lip}=\rg=\rga.\]
\end{theorem}

\begin{proof}
The first inequality in \cref{EqRad1} is the second inequality in \cref{T1.2-1} in \cref{T1.2}.
The equalities in Asplund spaces are direct consequences of \cref{Lem1}.
The remaining second inequality in \cref{EqRad1} follows from \cref{L3.1}  together with the definition \cref{rad} of the radius of metric regularity.
Indeed, by \cref{L3.1}, there is a function $f\in\mathcal{F}_{lip}$ such that
$\lip f(\bx)\leq\rga$, and
$0\leq {\rm rg\,}(F+f)(\bx,\by)\leq {\rm rg}^+(F+f)(\bx,\by){=0}$, i.e., $F+f$ is not metrically regular at $(\bx,\by)$, and therefore $\Rad{R}{\lip}\leq\lip f(\bx)\leq \rga$.
Finally, the last assertion follows from $\Rad{sR}{\lip}\leq \Rad{R}{\lip}=\rg$ together with the bound $\Rad{sR}{\lip}\geq \rg$ in \cref{T1.2-1a} in \cref{T1.2}.
\qed\end{proof}

\section{Proof of \cref{L3.1}}
\label{S4}

%Let $\rga<+\infty$.
We assume that $\rga>0$; otherwise the statement trivially holds with $f\equiv0$.
By definition \eqref{rga}, there exist sequences %\magenta{$\gF\ni(x_k,y_k)\to(\bx,\by)$},
$\gph F\ni (x_k,y_k)\to(\bx,\by)$,
$\eps_k\downarrow0$,
$y_k^*\in\Sp_{Y^*}$, and
$x_k^*\in D^*_{\eps_k}F(x_k,y_k)(y_k^*)$
such that
\begin{align}
\label{T4.1P1}
\rga = \lim_{k\to\infty}\norm{x_k^*}.
\end{align}
Consider any real $\ga>\rga$.
%By passing to subsequences,
Without loss of generality,
we assume that
\begin{align}
\label{T4.1P1+}
\inf_{k\in\N}\norm{x_k^*}>0\AND
\sup_{k\in\N}\norm{x_k^*}<\ga.
\end{align}
We have to consider two cases.

{\bf\em Case 1:}
$x_k\ne\bx$ for infinitely many $k\in\N$.
Denote $t_k:=\|x_k-\bx\|$.
Thus, $t_k\to0$ as $k\to\infty$.
By passing to subsequences and then relabeling appropriately, we can ensure that
\if{$t_k>0$
for all $k\in\N$, and sequences $\{t_k\}$ and $\{t_k-t_{k+1}\}$ are strictly decreasing.
}\fi
$0<t_{k+1}<t_k/2$ for all $k\in\N$. 
As a consequence, $t_k-t_{k+1}>t_{k+1}>t_{k+1}-t_{k+2}$.
Hence, $\{t_k\}$ and $\{t_k-t_{k+1}\}$ are both strictly decreasing sequences of positive numbers.
For each $k\in\N$, define $\rho_k:=(t_k-t_{k+1})/2$.
Thus, $\{\rho_k\}$ is also a strictly decreasing sequence of positive numbers.
%Let $i,k\in\N$ and $i>k$.
Moreover,
\begin{align}
\label{T4.1P4-}
\overline B_{t_k+\rho_k}(\bx)\cap \overline B_{\rho_i}(x_i)=\es \qdtx{for all}i,k\in\N,\;i<k.
\end{align}
Indeed,
let $i,k\in\N$ and $i<k$. Then
\begin{multline*}
\inf_{x\in\overline B_{t_k+\rho_k}(\bx),\; x'\in\overline B_{\rho_i}(x_i)}\|x-x'\|\ge \|x_i-\bx\|-\sup_{x\in\overline B_{t_k+\rho_k}(\bx)}\|x-\bx\|-\sup_{x'\in\overline B_{\rho_i}(x_i)}\|x'-x_i\|
\\
\ge t_i-(t_k+\rho_k)-\rho_i
\ge t_{i}-(t_{i+1}+\rho_{i+1})-\rho_i= \rho_{i}-\rho_{i+1}>0.
\end{multline*}
Since $\overline{B}_{\rho_k}(x_k)\subset \overline{B}_{t_k+\rho_k}(\bx)$, it follows from \eqref{T4.1P4-} that
\begin{align}
\label{T4.1P4}
{\overline{B}_{\rho_k}(x_k)\cap \overline{B}_{\rho_i}(x_i)=\es}
%B_{\rho_k}(x_k)\cap B_{\rho_i}(x_i)=\es
\qdtx{for all}i\ne k.
\end{align}
For each $k\in\N$, choose a point $v_k\in\Sp_{Y}$ such that
\begin{align}
\label{T4.1P5}
\ang{y_k^*,v_k}>1-1/k.
\end{align}
For all $k\in\N$ and $x\in X$, set
\begin{subequations}
\label{Eqf_k_LipPerm}
\begin{align}
\label{T4.1P5+}
s_{k}(x):=& \max\left\{1-\left(\|x-x_k\|/\rho_k\right)^{1+\frac1k}, 0\right\},
\\
\label{T4.1P6}
g_{k}(x):=&\ang{x_k^*,x-x_k}v_k,
\\
\label{T4.1P7}
f_{k}(x):=&s_{k}(x)g_{k}(x).
\end{align}
\end{subequations}
Observe
%from \cref{T4.1P5} and \cref{T4.1P7}
that $s_{k}(x)=0$ and $f_{k}(x)=0$ for all $x\notin B_{\rho_k}(x_k)$.
In view of \cref{T4.1P4},
the function
\begin{equation}
\label{Eqf_LipPerm}
f(x):=-\sum_{k=1}^\infty f_k(x),\quad x\in X,
\end{equation}
is well defined,
$f(x)=-f_k(x)$ for all $x\in B_{\rho_k}(x_k)$ and all $k\in\N$, and $f(x)=0$ for all ${x\notin\cup_{k=1}^\infty B_{\rho_k}(x_k)}$.
In particular, $f(\bx)=0$.
Observing that
$s_{k}(x_k)=1$,
$g_{k}(x_k)=0$, and
the function $s_{k}$ is differentiable at $x_k$ with
$\nabla s_{k}(x_k)=0$, we have
\begin{align}
\label{T4.1P13}
f(x_k)=0\AND
\nabla f(x_k)=-\ang{x_k^*,\cdot}v_k
\qdtx{for all}k\in\N.
\end{align}
Given any $x,x'\in\overline{B}_{\rho_k}(x_k)$, by the mean-value theorem applied to the function $t\mapsto t^{1+\frac1k}$ on $\R_+$, there is a number $\theta\in[0,1]$ such that
\begin{multline*}
s_{k}(x)-s_{k}(x')= -\left(1+1/k\right)\rho_k^{-(1+\frac1k)} (\theta\|x-x_k\|
\\
+(1-\theta)\|x'-x_k\|)^{\frac1k} (\|x-x_k\|-\|x'-x_k\|),
\end{multline*}
and consequently, assuming without loss of generality that $\|x-x_k\|\le\|x'-x_k\|$, we have
\begin{align*}
%\label{T4.1P8}
|s_{k}(x)-s_{k}(x')|\le \left(1+1/k\right)\rho_k^{-(1+\frac1k)} \|x'-x_k\|^{\frac1k}\|x-x'\|.
\end{align*}
If $x\ne x'$, then
\begin{align}
\notag
&\frac{\|f_{k}(x)-f_{k}(x')\|}{\|x-x'\|}\le \frac{|s_{k}(x)-s_{k}(x')|}{\|x-x'\|}\|g_k(x)\|+ s_{k}(x')\frac{\|g_k(x)-g_k(x')\|}{\|x-x'\|}
\\
\notag
&\qquad\le\left(1+1/k\right) \left({\|x'-x_k\|}/{\rho_k}\right)^{1+\frac1k} \|x_k^*\|
+\left(1-\left({\|x'-x_k\|}/{\rho_k}\right)^{1+\frac1k}\right) \|x_k^*\|
\\
\label{T4.1P11}
&\qquad=\left(\frac1k\left( {\|x'-x_k\|}/{\rho_k}\right)^{1+\frac1k}+1\right) \|x_k^*\|
\le\left(1+1/k\right)\|x_k^*\|.
\end{align}
By \cref{T4.1P1+}, there is a number $\hat k\in\N$ such that
$\left(1+1/ k\right)\|x_k^*\|<\ga$ for all $k>\hat k$.
Thus, for any $k> \hat k$, the function $f_{k}$ is Lipschitz continuous on $\overline{B}_{\rho_k}(x_k)$ with modulus {less than} $\ga$.
As a consequence, we also have
\begin{align}
\label{T4.1P12}
\|f_k(x)\|<\ga(\rho_k-\|x-x_k\|)
\qdtx{for all}x\in B_{\rho_k}(x_k).
\end{align}
Indeed,
given any $x\in{B}_{\rho_k}(x_k)$ with $x\ne x_k$, we set $x':=x_k+\rho_k\frac{x-x_k}{\|x-x_k\|}\in\overline B_{\rho_k}(x_k)$.
Then $f_k(x')=0$, and consequently, $\|f_k(x)\|<\ga\|x'-x\|= \ga(\rho_k-\|x-x_k\|)$.
If $x=x_k$, then, in view of \cref{T4.1P13}, inequality \cref{T4.1P12} holds true trivially.

{\em Claim 1:}
For any $k>\hat k$,
the function $f$ is Lipschitz continuous on $B_{t_k+\rho_k}(\bx)$
with modulus~$\ga$.

Indeed, let $k>\hat k$, $x,x'\in B_{t_k+\rho_k}(\bx)$ and $x\ne x'$.
\\
1) If $x,x'\in B_{\rho_i}(x_i)$ for some $i\ge k$, then $\|f(x)-f(x')\|<\ga\|x-x'\|$ since $f=-f_i$ on $B_{\rho_i}(x_i)$.
\\
2) If $x\in B_{\rho_i}(x_i)$ and $x'\in B_{\rho_j}(x_j)$ for some $i,j\ge k$ with $i\ne j$, then, thanks to \cref{T4.1P4}, we have $\|x_i-x_j\|\ge\rho_i+\rho_j$, and using \cref{T4.1P12}, we obtain
\begin{align*}
%\label{T4.1P8}
\|f(x)-f(x')\|&\le\|f(x)\|+\|f(x')\|< \ga(\rho_i+\rho_j-\|x-x_i\|-\|x'-x_j\|)
\\&\le
\ga(\|x_i-x_j\|-\|x-x_i\|-\|x'-x_j\|)
\le
\ga\|x-x'\|.
\end{align*}
3) If $x\in B_{\rho_i}(x_i)$ for some $i\ge k$, and $x'\notin\bigcup_{j=k}^\infty B_{\rho_j}(x_j)$, then $\|x'-x_i\|\ge\rho_i$, and, thanks to \eqref{T4.1P4-}, we also have $x'\notin\bigcup_{j=1}^{k-1}B_{\rho_j}(x_j)$, implying $f(x')=0$.
Using \cref{T4.1P12}, we obtain
\begin{align*}
%\label{T4.1P8}
\|f(x)-f(x')\|
%=\|f(x)\|&\le\ga\,''(\rho_i-\|x-x_i\|)
<
\ga(\|x'-x_i\|-\|x-x_i\|)
\le
\ga\|x-x'\|.
\end{align*}
4) If $x,x'\notin\cup_{j=k}^\infty B_{\rho_j}(x_j)$, then $f(x)=f(x')=0$.
\\
Thus, in all cases,
$\|f(x)-f(x')\|<\ga\|x-x'\|$.
\xqed
\smallskip

Hence, $f\in{\mathcal F}_{lip}$ with $\lip f(\bx)\leq \ga$ and, since $\ga$ can be chosen arbitrarily close to $\rga$, we have $\lip f(\bx)\leq\rga$.
\smallskip

{\em Claim 2:}
The function $f$ is Lipschitz rank one on $X\setminus\{\bx\}$.

Indeed, let $x\in X\setminus\{\bx\}$.
If $x\in\overline{B}_{\rho_k}(x_k)$ for some $k\in\N$, then, thanks to \cref{T4.1P4,T4.1P6,T4.1P7,Eqf_LipPerm}, there is a $\tilde\rho_k>\rho_k$ such that $f(u)=-s_k(u)\ang{x_k^*,u-x_k}v_k$ for all $u\in B_{\tilde\rho_k}(x_k)$.
Note that $B_{\tilde\rho_k}(x_k)$ is a \nbh\ of $x$.
If $x\notin\cup_{k\in\N}\overline{B}_{\rho_k}(x_k)$, then, thanks to \eqref{T4.1P4}, $f(u)=0$ for all $u$ in a sufficiently small neighborhood of $x$.
Hence, $f$ is Lipschitz rank one on $X\setminus\{\bx\}$.
\xqed
\smallskip

Next, consider any $\alpha\in[0,1]$.
In view of \cref{T4.1P13}, for all $k\in\N$, we have
$D^*f(x_k)(y_k^*)=-\ang{y_k^*,v_k}x_k^*$.
Thanks to the first inclusion in \cref{T2.3}, we have
$$\left(1-\alpha\ang{y_k^*,v_k}\right)x_k^*\in D^*_{\eps_k}F(x_k,y_k)(y_k^*)+D^*(\alpha f)(x_k)(y_k^*)\subset D^*_{\eps_k'}(F+\alpha f)(x_k,\by)(y_k^*),$$
where
$\eps_k':=(\alpha\|\nabla f(x_k)\|+1)\eps_k$.
Thanks to \cref{T4.1P13,T4.1P1+},
$\eps_k'=(\alpha\|x_k^*\|+1)\eps_k<(\ga+1)\eps_k$.
Thus, $\eps_k'\to 0$ as $k\to\infty$ and, in view of \cref{T4.1P5}, we conclude that
%\sloppy
\[{\rm rg}^+(F+\alpha f)(\bx,\by)\leq\lim_{k\to\infty} \left(1-\alpha(1-1/k)\right)\norm{x_k^*}= (1-\alpha)\rga.\]
This completes the proof in the first case.
\smallskip

{\bf\em Case 2:}
$x_k\ne\bx$ for not more than finitely many $k\in\N$.
Then $x_k=\bx$ for all sufficiently large $k$ and, replacing the sequences by their tails, we can assume that $x_k=\bx$ for all $k\in\N$.
We are going to reduce this case to the previous one.
For that,
we now construct new sequences
$\gph F\ni(\tilde x_k,\tilde y_k)\to(\bx,\by)$ and $\tilde\eps_k\downarrow0$ such that
$\tilde x_k\ne\bx$ and
$x_k^*\in D^*_{\tilde \eps_k}(\tilde x_k,\tilde y_k)(y_k^*)$
for all $k\in\N$.
\smallskip

{\em Claim 3:}
There are not more than finitely many $k\in\N$ with the property
\begin{equation}
\label{EqClaim1}
\exists \rho>0:\quad \gph F\cap B_{\rho}(\bx,y_k)\subset\{\bx\}\times Y.
\end{equation}

Indeed, if \eqref{EqClaim1} is true with some $\rho>0$ for some $k\in\N$, then, by the definition of the $\eps$-coderivative, we can find a $\tilde\rho\in(0,\rho)$ such that, for all $(x,y)\in\gph F\cap B_{\tilde\rho}(\bx,y_k)\setminus\{(\bx,y_k)\}$, we have
\[-\ang{y_k^*,y-y_k}=\ang{x_k^*,x-\bx}-\ang{y_k^*,y-y_k} <\eps_k(\norm{x-\bx}+\norm{y-y_k}),\]
%where $\eps_k':=\eps_k+1/k$,
verifying that
$0\in D^*_{\eps_k}F(\bx,y_k)(y_k^*)$.
If there were infinitely many $k\in\N$ fulfilling \eqref{EqClaim1}, this would contradict the assumption that $\rga>0$.
\xqed
\smallskip

Thus,
%Claim 1 is true, and
we can assume that \eqref{EqClaim1} fails for all $k\in\N$.
%i.e., for any $k\in\N$ and $\rho>0$, there is a point $(x,y)\in\gph F\cap B_{\rho}(\bx,y_k)$ with $x\ne\bx$.
%%%%%%%%%%%%%%%%%%%%%%%%%%%%%%%%%%%
\if{
 Now define for every $k$ the quantity
\[\sigma_k:=\limsup_{\stackrel{\gph F\ni (x,y)\to(\bx,y_k)}{ x\not=\bx}}\frac {\norm{x-\bx}}{\norm{y-y_k}}\in[0,\infty].\]
{\em Claim 2:} $\liminf_{k\to\infty}\sigma_k=:2\sigma>0$.

Assume on the contrary that $\sigma=0$. By passing to a subsequence, we can assume that $\lim_{k\to\infty}\sigma_k=0$. For every $k$ we can find some radius $\rho_k>0$ such that
for every $(x,y)\in B_{\rho_k}(\bx,y_k)\cap \gph F$ we have
\begin{align*}\norm{x-\bx}\leq \left(\sigma_k+\frac 1k\right)\norm{y-y_k}\ \mbox{and}\ \ang{x_k^*,x-\bx}-\ang{y_k^*,y-y_k}\leq \left(\eps_k+\frac 1k\right)(\norm{x-\bx}+\norm{y-y_k})
\end{align*}
implying
\begin{align*}
-\ang{y_k^*,y-y_k}&\leq \left(\eps_k+\frac 1k\right)(\norm{x-\bx}+\norm{y-y_k})+\norm{x_k^*}\norm{x-\bx}\\
&\leq \left(\eps_k+\frac 1k +\left(\sigma_k+\frac 1k\right)\bar\ga\right)(\norm{x-\bx}+\norm{y-y_k}).
\end{align*}
Thus $0\in D^*_{\eps_k'}F(\bx,y_k)(y_k^*)$ with $\eps_k':=\eps_k+1/k+(\sigma_k+1/k)\bar\ga\to 0$ as $k\to\infty$ and the contradiction $\rga=0$ follows.

Hence Claim 2 holds true as well and we can assume that $\sigma_k>\sigma$ $\forall k$.
}\fi
%%%%%%%%%%%%%%%%%%%%%%%%%%%%%%%%
%Setting $\eps_k':=\eps_k+1/k$, we can find for every
Choose any number $k\in\N$ and find a number $\rho_k\in (0,1/k)$ such that, for all $(x,y)\in\gph F\cap \overline B_{\rho_k}(\bx,y_k)\setminus\{(\bx,y_k)\}$, it holds
\begin{equation}
\label{EqAuxBnd1}
  \ang{x_k^*, x-\bx}-\ang{y_k^*,y-y_k}< \eps_k(\norm{x-\bx}+\norm{y-y_k}).
\end{equation}
%where $\eps_k':=\eps_k+1/k$.
Next, we define the function $\varphi_k:(0,\rho_k]\to\R$ by
\begin{gather}
\label{L3.1P6}
\varphi_k(\rho):= \sup\set{ \frac{\ang{x_k^*,x-\bx}-\ang{y_k^*,y-y_k}- \eps_k\norm{y-y_k}}{\norm{x-\bx}}\Bmv \begin{array}{l}
(x,y)\in B_\rho(\bx,y_k)\\
x\ne\bx,\;
y\in F(x)
\end{array}}.
\end{gather}
%%%%%%%%%%%%%%%%%%%%%%%%%%%%%%%%%%%%%%%%%%%%%%%%%%%%%%%%%%%%%%%%%%%
\if{It follows that  for every $\rho\in(0,\rho_k]$ we have
\begin{align*}\varphi_k(\rho)&\geq \sup\set{\frac{-\norm{x_k^*}\norm{x-\bx}-(\norm{y_k^*}+\eps_k')\norm{y-y_k}}{\norm{x-\bx}}\Bmv \begin{array}{l}(x,y)\in\gph F\cap B_\rho(\bx,y_k),\\ x\ne \bx\end{array}}\\
&\geq -\norm{x_k^*}-\frac{1+\eps_k'}\sigma.
\end{align*}
}\fi
%%%%%%%%%%%%%%%%%%%%%%%%%%%%%%%%
By
%Claim 1
the assumption,
$\varphi_k(\rho)>-\infty$ for all $\rho\in(0,\rho_k]$.
On the other hand, condition \cref{EqAuxBnd1} implies $\varphi_k(\rho_k)\leq \eps_k$.
Since
%for every $0<\rho'<\rho''\leq\rho_k$ we have $\varphi_k(\rho')\leq \varphi_k(\rho'')$, for every $k$
$\varphi_k$ is nondecreasing on $(0,\rho_k]$,
we have
\begin{gather}
\label{L3.1P7}
%-\infty\le
\varsigma_k:=\inf_{\rho\in (0,\rho_k]}\varphi_k(\rho)=\lim_{\rho\downarrow 0}\varphi_k(\rho)<\eps_k.
\end{gather}

{\em Claim 4: } $\lim_{k\to\infty}\varsigma_k=0$.

Since
%$\limsup_{k\to\infty}\varsigma_k\leq \limsup_{k\to\infty}\eps_k'=0$, there
$\eps_k\downarrow0$, in view of \cref{L3.1P7}, it
remains to show that $\liminf_{k\to\infty}\varsigma_k\geq0$.
Assume on the contrary that $\liminf_{k\to\infty}\varsigma_k<0$, i.e., for some number $\eta>0$ and a subsequence of $\varsigma_k$, it holds, without relabeling, that $\varsigma_k<-\eta$ for all $k\in\N$.
For every $k\in\N$,
by \cref{L3.1P7},
we can find a $\tilde\rho_k\in(0,\rho_k]$ with $\varphi_k(\tilde\rho_k)<-\eta$, i.e.,
by definition \cref{L3.1P6},
for all
$(x,y)\in\gph F\cap B_{\tilde\rho_k}(\bx,y_k)$ with $x\ne\bx$, we have
\begin{gather}
\label{L3.1P8}
\ang{x_k^*,x-\bx}-\ang{y_k^*,y-y_k}< \eps_k\norm{y-y_k}-\eta\norm{x-\bx}.
\end{gather}
Moreover, thanks to \cref{EqAuxBnd1}, condition \cref{L3.1P8} is satisfied also for all
$(x,y)\in\gph F\cap B_{\tilde\rho_k}(\bx,y_k)$ with $x=\bx$ and $y\ne y_k$.
Hence, in view of the first inequality in \cref{T4.1P1+}, we have for all
$(x,y)\in\gph F\cap B_{\tilde\rho_k}(\bx,y_k)\setminus\{(\bx,y_k)\}$:
\begin{gather*}
%\label{L3.1P8}
\ang{\big(1-\eta/{\norm{x_k^*}}\big)x_k^*,x-\bx} \leq \ang{x_k^*,x-\bx}+\eta\norm{x-\bx}< \ang{y_k^*,y-y_k}+\eps_k\norm{y-y_k}.
\end{gather*}
Thus, $\big(1-\eta/\norm{x_k^*}\big) x_k^*\in D_{\eps_k}^*F(\bx,y_k)(y_k^*)$ for all $k\in\N$, and, by the definition, we obtain
\[\rga\leq\liminf_{k\to\infty} \big(1-\eta/\norm{x_k^*}\big) \norm{x_k^*}=\liminf_{k\to\infty}\norm{x_k^*} -\eta,\]
which contradicts \cref{T4.1P1}.
\xqed
\smallskip

%Thus, Claim 2 is true.
%In particular,
%We can assume that $|\varsigma_k|<+\infty$ for all $k\in\N$.
Fix any number $k\in\N$.
Next, we choose a number $\tilde\rho_k\in(0,\rho_k/2)$,
%with $\varphi_k(\tilde\rho_k)<\varsigma_k+1/k$,
and then some $(\hat x_k,\hat y_k)\in\gph F\cap B_{\tilde\rho_k}(\bx,y_k)$ with
$\hat x_k\ne\bx$ and
\begin{equation}
\label{EqAuxBnd2}
\frac{\ang{x_k^*,\hat x_k-\bx}-\ang{y_k^*,\hat y_k-y_k}-\eps_k\norm{\hat y_k-y_k}}{\norm{\hat x_k-\bx}}>\varphi_k(\tilde\rho_k)-\frac1k\ge \varsigma_k-\frac1k.
\end{equation}
\if{
implying
\[\ang{x_k^*,\hat x_k-\bx}-\ang{y_k^*,\hat y_k-y_k}-\eps_k'(\norm{\hat x_k-\bx}+\norm{\hat y_k-y_k})> (\varphi_k(\tilde\rho_k)- 1/k-\eps_k')\norm{\hat x_k-\bx}.\]
}\fi
Consider the (continuous) function $\psi_k:\gph F\cap\overline B_{\rho_k}(\bx,y_k)\mapsto\R$ given by
\begin{gather}
\label{L3.1P12}
\psi_k(x,y):=
\eps_k\norm{(x,y)-(\bx,y_k)}-\ang{x_k^*, x-\bx}+\ang{y_k^*,y-y_k}.
\end{gather}
By \eqref{EqAuxBnd1}, $\psi_k(x,y)\ge0$ for all $(x,y)\in\gph F\cap\overline B_{\rho_k}(\bx,y_k)$, while from \eqref{EqAuxBnd2} we obtain:
\[\psi_k(\hat x_k,\hat y_k)<
%\inf \psi_k +
\eps_k'\norm{\hat x_k-\bx},\]
where $\eps_k':=\eps_k+1/k-\varsigma_k$.
Observe that $\eps_k'>0$, and $\lim_{k\to\infty}\eps_k'=0$ by Claim 4.
By Ekeland's variational principle (see, e.g., \cite[Theorem~2.26]{Mor06.1}), we find a point $(\tilde x_k,\tilde y_k)\in\gph F\cap\overline B_{\rho_k}(\bx,y_k)$ such that
$\psi_k(\tilde x_k,\tilde y_k)\leq \psi_k(\hat x_k,\hat y_k)$, and
\begin{gather}
\label{EqEkeland2}
\norm{(\tilde x_k,\tilde y_k)-(\hat x_k,\hat y_k)}\leq\norm{\hat x_k-\bx},\\
\label{EqEkeland3}
\psi_k(x,y)+\eps_k'\norm{(x,y)-(\tilde x_k,\tilde y_k)}>\psi_k(\tilde x_k,\tilde y_k)
\end{gather}
for all $(x,y)\in\gph F\cap \overline B_{\rho_k}(\bx,y_k) \setminus\{(\tilde x_k,\tilde y_k)\}$.
From \eqref{EqEkeland2}, we obtain:
\begin{align*}
\norm{(\tilde x_k,\tilde y_k)-(\bx,y_k)}&\leq \norm{(\hat x_k,\hat y_k)-(\bx,y_k)}+\norm{(\tilde x_k,\tilde y_k)-(\hat x_k,\hat y_k)}
\\&\leq
\norm{(\hat x_k,\hat y_k)-(\bx,y_k)}+\norm{\hat x_k-\bx}
\\&\leq
2\norm{(\hat x_k,\hat y_k)-(\bx,y_k)}\leq 2\tilde\rho_k<\rho_k.
\end{align*}
Finally, using \cref{L3.1P12,EqEkeland3}, we have for all $(x,y)\in\gph F\cap\overline B_{\rho_k}(\bx,y_k)\setminus\{(\tilde x_k,\tilde y_k)\}$:
\begin{align*}
\ang{x_k^*,x-\tilde x_k}-&\ang{y_k^*,y-\tilde y_k}= \ang{x_k^*,x-\bx}-\ang{y_k^*,y-y_k}- \ang{x_k^*,\tilde x_k-\bx}+\ang{y_k^*,\tilde y_k-y_k}
\\&=
\eps_k\big(\norm{(x,y)-(\bx,y_k)}-\norm{(\tilde x_k,\tilde y_k)-(\bx,y_k)}\big)-\psi_k(x,y)
+\psi_k(\tilde x,\tilde y)
\\&<
(\eps_k+\eps_k')\norm{(x,y)-(\tilde x_k,\tilde y_k)}.
\end{align*}
Since
%$B_{\rho_k/4}(\tilde x_k,\tilde y_k)\subset
$(\tilde x_k,\tilde y_k)\in\Int
\overline B_{\rho_k}(\bx,y_k)$, we have $x_k^*\in D_{\tilde \eps_k}^*F(\tilde x_k,\tilde y_k)(y_k^*)$, where $\tilde\eps_k:=\eps_k+\eps_k'\to 0$ as $k\to\infty$.
Thus, we have constructed the desired sequences $(\tilde x_k,\tilde y_k)$ and $\tilde \eps_k$; hence, Case 2 reduces to Case 1, and the proof of \cref{L3.1} is complete.
\qed

\begin{remark}
The above proof is constructive.
The function $f$ with the desirable properties is defined by formulas \cref{Eqf_k_LipPerm,Eqf_LipPerm} which involve special sequences $\{x_k\}$, $\{\rho_k\}$ and $\{v_k\}$.
The procedure adopted here follows that used in \cite{GfrKru}.
It is not unique.
One could try to adjust the techniques used, e.g., in \cite[Proof of Lemma 2]{BarFabKol22}.
\end{remark}

\section*{Declarations}

\noindent{\bf Funding. }The second author benefited
from the support of the Australian Research Council, project DP160100854, and the European Union's Horizon 2020 research and innovation
programme under the Marie Sk{\l }odowska--Curie Grant Agreement No. 823731
CONMECH.

\noindent{\bf Conflict of interest.} The authors have no competing interests to declare that are relevant to the content of this article.

\noindent{\bf Data availability. }
Data sharing is not applicable to this article as no datasets have been generated or analysed during the current study.

\section*{Acknowledgement}

The authors wish to thank the referees for their unique dedication and hard work reading the manuscript, checking every detail, and making impressive effort to help us improve the text.
The paper has indeed strongly benefited from the comments and suggestions of the referees.

Many thanks to the Editor-in-Chief for the perfect choice of the referees and overall handling of our manuscript.

\bibliographystyle{spmpsci}
\bibliography{buch-kr,kruger,kr-tmp}

\end{document}